\documentclass[a4paper,9pt]{article}

\usepackage{lineno,hyperref}
\usepackage{amssymb}
\usepackage{amsmath}
\usepackage{amsthm}
\usepackage{graphics,graphicx}
\usepackage{textcomp}
\usepackage{paralist}
\usepackage{mathtools}
\usepackage[all,arc,poly]{xy}
\usepackage[ansinew]{inputenc}
\usepackage{xspace}
\usepackage{todonotes}
\usepackage{txfonts}
\usepackage[sc]{mathpazo} 
\usepackage{mathrsfs}
\usepackage[all,arc,poly]{xy}
\usepackage{afterpage}
\usepackage{graphicx}

\newcommand{\N}{\mathbb{N}}
\newcommand{\Z}{\mathbb{Z}}

\newcommand{\nc}{\newcommand}
\nc{\look}{\marginpar{$\bullet$}}
\nc{\FO}{\ensuremath{\mathsf{FO}}\xspace}
\nc{\biFO}{\ensuremath{\mathsf{FO/\!\!\sim}}\xspace}
\nc{\MSO}{\ensuremath{\mathsf{MSO}}\xspace}
\nc{\MPL}{\ensuremath{\mathsf{MPL}}\xspace}
\nc{\Lmu}{\ensuremath{\mathsf{L}_\mu}\xspace}
\nc{\CTL}{\ensuremath{\mathsf{CTL}^*}\xspace}
\nc{\LTL}{\ensuremath{\mathsf{LTL}}\xspace}
\nc{\PDL}{\ensuremath{\mathsf{PDL}}\xspace}
\nc{\WCL}{\ensuremath{\mathsf{WCL}}\xspace}
\nc{\ML}{\ensuremath{\mathsf{ML}}\xspace}
\nc{\gML}{\ensuremath{\mathsf{ML}^\forall}\xspace}
\nc{\CK}{\ensuremath{\mathsf{CK}}\xspace}
\nc{\MLCK}{\ensuremath{\mathsf{ML}[\mathsf{CK}]}\xspace}
\nc{\MLPA}{\ensuremath{\mathsf{ML}[\mathsf{PA}]}\xspace}
\nc{\MLCP}{\ensuremath{\mathsf{ML}[\mathsf{CK,PA}]}\xspace}
\nc{\RC}{\ensuremath{\mathsf{RC}}\xspace}
\nc{\Sfive}{\ensuremath{\mathsf{S5}}\xspace}
\nc{\RCs}{\ensuremath{\mathsf{RC^\sim}}\xspace}
\nc{\RCPA}{\ensuremath{\mathsf{RC}[\mathsf{PA}]}\xspace}
\nc{\FOCK}{\ensuremath{\mathsf{FO}[\mathsf{CK}]}\xspace}
\nc{\FORC}{\ensuremath{\mathsf{FO}[\mathsf{RC}]}\xspace}
\nc{\FORCs}{\ensuremath{\mathsf{FO}[\mathsf{RC^\sim}]}\xspace}
\nc{\EF}{Ehrenfeucht-Fra\"iss\'e\xspace}

\newcommand{\ca}{Cayley\xspace}
\newcommand{\sh}{\mathrm{short}}

\newcommand{\gen}{\mathrm{gen}}

\nc{\str}[1]{{\mathfrak{#1}}}
\nc{\restr}{\!\restriction\!}
\nc{\G}{\mathbb{G}}
\nc{\HH}{\mathbb{H}}
\nc{\VV}{\mathbb{V}}
\nc{\abar}{\mathbf{a}}
\nc{\bbar}{\mathbf{b}}
\nc{\cbar}{\mathbf{c}}
\nc{\xbar}{\mathbf{x}}
\nc{\ybar}{\mathbf{y}}
\nc{\zbar}{\mathbf{z}}
\nc{\ubar}{\mathbf{u}}
\nc{\sbar}{\mathbf{s}}
\nc{\tbar}{\mathbf{t}}
\nc{\vbar}{\mathbf{v}}
\nc{\wbar}{\mathbf{w}}
\nc{\ellp}{{\ell+1}}
\nc{\ellm}{{\ell-1}}
\nc{\kp}{{k+1}}
\nc{\km}{{k-1}}
\nc{\jp}{{j+1}}
\nc{\jm}{{j-1}}
\nc{\ip}{{i+1}}
\nc{\im}{{i-1}}
\nc{\brck}[1]{[\![ #1 ]\!]}
\nc{\PI}{\ensuremath{\mbox{\textbf{I}}}\xspace}
\nc{\PII}{\ensuremath{\mbox{\textbf{II}}}\xspace}

\newtheorem{theorem}{Theorem}[section]
\newtheorem*{theorem*}{Theorem}

\newtheorem{corollary}[theorem]{Corollary}
\newtheorem{lemma}[theorem]{Lemma}

\newtheorem{pprov}[theorem]{Proviso}

\theoremstyle{definition}
\newtheorem{definition}[theorem]{Definition}

\newtheorem{example}[theorem]{Example}
\newtheorem{prop}[theorem]{Proposition}

\newtheorem{remark}[theorem]{Remark}
\newtheorem{obs}[theorem]{Observation}

\theoremstyle{remark}

\begin{document}

\title{Some Structure Theory for Cayley Graphs and Associated Hypergraphs}
\author{Felix Canavoi\\
Department of Mathematics\\
Technische Universit\"at Darmstadt}
\maketitle

\begin{abstract}
We expand the structure theory of finite \ca graphs
that avoid specific cyclic coset patterns.
A focus lies on the exploration of duality in
related structures and associated hypergraphs, especially
applied to the local analysis of paths and cycles.
We present several characterisations of local tree-likeness
for these structures and show a close connection
to $\alpha$-acyclicity of hypergraphs.
\end{abstract}

\tableofcontents

\section{Introduction}\label{sec:intro}

Acyclic, discrete structures play a significant role
in computer science. Many algorithmic graph problems that
are hard in general become tractable for trees.
These efficient
algorithms can be further adapted to larger classes
of graphs, like graphs of bounded tree-width
\cite{Co90}, \cite{Bo93}, \cite{RS12}.
These are graphs that are not necessarily trees,
but still structurally simple and in some sense
tree-like. Generalizing from graphs to hypergraphs,
there are several different notions of acyclicity,
like $\gamma$-, $\beta$- and $\alpha$-acyclicity
and tree decomposability,
that admit many different characterisations and
find applications in database theory,
constraint satisfaction problems
and finite model theory
\cite{BeeriFaginetal}, \cite{DP89}, \cite{OPS13},
\cite{DKV02}, \cite{GM99}, \cite{BDG07}.

This work focuses on the notion of coset acyclicity
for \ca graphs, a class of graphs that plays an important
role in discrete mathematics~\cite{Bi93}, \cite{CG96}, \cite{Li02},
combinatorics~\cite{Ba79},
information theory~\cite{KRY09}, coding theory~\cite{MBG09}, \cite{DEL21},
and network theory~\cite{He97}, \cite{LJD93}
among other fields of mathematics.
Coset acyclicity was in introduced
in~\cite{Otto12JACM} to construct certain finite hypergraph
coverings that have an arbitrarily high degree of
$\alpha$-acyclicity for a model-theoretic
characterisation theorem in the vein of the
van Benthem-Rosen theorem \cite{Benthem83}, \cite{Rosen}.
In~\cite{CO17},~\cite{Ca18} and~\cite{CO21} coset acyclic \ca graphs were used
to cover transition systems directly in order
to prove further model-theoretic characterisation
theorems.
These results build upon the work in~\cite{DawarOttoAPAL09}
and further develop the model-theoretic techniques
that were used there.
The current work presents the graph structure theory,
which was essentially used
in \cite{CO17}, \cite{Ca18}, \cite{CO21},
in greater detail and in a self-contained
manner that makes it accessible for wider ranging applications
beyond model-theory. It is the goal to provide
a general toolbox for the analysis of coset patterns
in finite \ca graphs that can be regarded as locally tree-like.

Coset acyclicity generalises the ordinary graph-theoretic
notion of a cycle in the context of \ca graphs. In a
\ca graph, every edge is induced by an element~$e$ from
a generating set~$E$ of the associated \ca group.
A single step can therefore be represented by
a single generator. Coset cycles generalise this
notion by combining several generator steps into
a larger step that is represented by a coset that
is generated by a subset $\alpha\subseteq E$
of generators, compared to a single generator $e\in E$.
The formal definition of coset cycles that stipulates
precisely which sequences of cosets form a coset cycle
leads to a nice structure theory for coset acyclic
\ca graphs and \ca graphs without short coset cycles~\cite{Otto12JACM}.
We further investigate the structure theory of coset
acyclic \ca graphs and present several ways in which
these structures can be regarded as locally tree-like.
The other central concept, besides coset cycles,
is the notion of coset paths, which generalise graph-theoretic paths
in the same way that coset cycles generalise graph-theoretic cycles.
Among other results, we present a qualified uniqueness property
for coset paths in coset acyclic \ca graphs, and
establish several close connections between coset
acyclic \ca graphs and $\alpha$-acyclic hypergraphs.

\subsection*{Outline}
Section~\ref{sec:prelim} introduces the basic notions
and definitions: \ca graphs, cycles, paths, acyclicity,
and hypergraphs and $\alpha$-acyclicity. Section~\ref{sec:cycles}
presents the formal definition of coset cycles, takes a close
look at \ca graphs without coset cycles of length~2, which play
a special role, and establishes the first connections between
coset acyclicity and $\alpha$-acyclicity. Section~\ref{sec:paths}
contains the main results of this work. It introduces coset
paths, develops uniqueness properties for coset paths in
acyclic \ca graphs, and deepens the connection between
\ca graphs and hypergraphs with a focus on the equivalence
between two different notions of distance, one in \ca graphs
w.r.t.\ coset paths and one in hypergraphs.

\section{Preliminaries}\label{sec:prelim}

In this section, we introduce the main
objects we want to investigate, \ca graphs
and hypergraphs. We also present some basic
and well-known notions of acyclicity for
these structures, which will be further
developed and investigated in the course
of this work. We start with fixing some notation.

For an equivalence
relation~$R$ on~$A$, we denote the equivalence class
of an element $a\in A$ by $[a]_R$ and
write $A/{R}=\{[a]_R : a\in A\}$ for the set
of all equivalence classes.
The set of $R$-successors $\{b\in A : (a,b)\in R\}$
of an element~$a$ is denoted by $R[a]$.

\subsection{\ca graphs and acyclicity}

This work further investigates the notion
of coset acyclicity of \ca graphs, which was introduced
by Otto in~\cite{Otto12JACM}, and its connection to $\alpha$-acyclicity
of hypergraphs. This section introduces \ca graphs formally,
the usual graph-theoretic notion of acyclicity and
associated concepts. Coset acyclicity
is a generalisation
of the usual notion of graph acyclicity.

A \emph{Cayley group} is a group~$(G,\circ,1)$ with an
associated generator set~$E$ that consists of
non-trivial involutions, i.e.\ $e\neq 1$ and $e\circ e=1$,
for all $e\in E$.
That~$G$ is generated by the set~$E$ means that
every group element can be represented as a
product of generators. In other words, every
$g\in G$ can be represented as a word in~$E^*$;
w.l.o.g.\ such a representation is reduced
in the sense that is does not have any
factors~$e^2$.
We can view a non-empty generator set~$E$
as an alphabet and interpret any
word $v=e_1\dots e_n$ over~$E$ as a group element
in~$G$ via $[v]^G=e_1\circ\dots\circ e_n$.
We can also think of the letters~$e_i$ as the labels
of a path from~1 to~$[v]^G$ in the Cayley graph of~$G$.
For $v=e_1\dots e_n$, we denote by~$v^{-1}$ the word
$e_n\dots e_1$; since all generators are involutions,
$[v^{-1}]^G=([v]^G)^{-1}$.

\begin{definition}
 With every Cayley group~$(G,\circ,1)$
 generated by~$E$ one associates its \emph{\ca graph}
 $(G,(R_e)_{e\in E})$: its vertex set
 is the set of group elements~$G$, and
 its edge relations are
 $$R_e=\{\{v,v\circ e\}\in G\times G: v\in G\}.$$
 If~$G$ is a \ca group, we denote the group itself,
 its \ca graph and its set of group elements with~$G$.
 If~$G$ is a \ca graph, we also write $V[G]$
 for its vertex set and $R_e[G]$ for its $e$-labelled edge relation.
\end{definition}

In our case, all edge relations are loop-free,
undirected and complete matchings on~$G$.
Since~$E$ generates~$G$, the
graph $(G,(R_e)_{e\in E})$ is connected.
Furthermore, it is homogeneous in the sense that
every two vertices~$v$ and~$u$ are related by a
graph automorphism that is induced by multiplication
from the left with $uv^{-1}$.

For a subset $\alpha\subseteq E$
we consider the subgroup~$G_\alpha$, which is the subgroup
of~$G$ generated by the generators from~$\alpha$.
Its \ca graph, also denoted~$G_\alpha$, is a subgraph of~$G$;
it is isomorphic to the $\alpha$-component of~1.
The $\alpha$-component of an arbitrary group element~$v$
is described by its \emph{$\alpha$-coset}
$vG_\alpha=\{v\circ u\in G: u\in G_\alpha\}$.
Every $\alpha\subseteq E$ induces an equivalence relation
on~$G$ through partitioning~$G$ into its $\alpha$-cosets.
Hence, we usually denote the $\alpha$-coset of a group
element~$v$ as~$[v]_\alpha$.

The main notions that we investigate in this work
are paths and cycles. \ca graphs have
multiple edge relations~$R_e$ that are labelled
with generators $e\in E$ of the associate
\ca group. Hence, all paths and cycles will
be labelled with generators to differentiate the kind
of steps that lead from one vertex to the next.

\begin{definition}
 An \emph{($E$-labelled) path of length~$\ell$} in a \ca
 graph~$G$ is an alternating sequence $ v_1,e_1,v_2,\dots,v_\ell,e_\ell,v_\ellp $
 of vertices $v_i\in V[G]$ and labels $e_i\in E$
 such that $\{v_i,v_\ip\}\in R_{e_i}$,
 for all $1\leq i\leq\ell$, end all vertices and edges are distinct,
 with the possible exception of $v_1=v_\ellp$,
 in which case the path is called a \emph{cycle}.
 The vertices~$v_1$ and~$v_\ellp$ are called
 the endpoints of the path, and we speak of
 a path from~$v_1$ to~$v_\ellp$.
 If every edge of a path is labelled with an element
 from a subset $\alpha\subseteq E$, we call it an $\alpha$-path.
\end{definition}


The definition of paths leads to several well-known
notions like distance, reachability and connectedness:
The \emph{distance} $d(v,u)$ between two vertices~$v,u$ in a graph
is the minimal length of a path from~$v$ to~$u$;
$d(v,v)=0$ for all $v\in V$,
and $d(v,u)=\infty$ if there is no path from~$v$ to~$u$.
The \emph{$\ell$-neighbourhood} of a vertex~$v$, denoted $N^\ell(v)$,
is the set of vertices of distance at most~$\ell$ from~$v$,
i.e.\ $\{u: d(v,u)\leq\ell \}$. For $\alpha\subseteq E$,
a vertex~$u$ is \emph{$\alpha$-reachable}
from~$v$ if there is an $\alpha$-path from~$v$ to~$u$.

Graphs without cycles or
without any short cycles are structurally more simple.
This lends itself to be exploited by various applications like
efficient algorithms for generally intractable problems,
and it is important for model-theoretic constructions.
We will compare and generalise properties
of graphs without short cycles to graphs without short
coset cycles. We present some further notions connected
to cycles that will be important throughout.

\begin{definition}\label{dfn:kacyclic}
 Let $G$ be a \ca graph.
 \begin{enumerate}
  \item
    $G$ is \emph{acyclic} if it has no cycles.
  \item
    A $k$-cycle in~$G$ is a \emph{cycle of length~$k$}
    in~$G$.
  \item
    $G$ is \emph{$k$-acyclic} if it has no
    cycles of length $\leq k$.
  \item
    The \emph{girth} of~$G$ is the length
    of a minimal cycle.
  \item
    $G$ is a \emph{tree} if it is acyclic.
 \end{enumerate}
\end{definition}

Usually, trees are defined as acyclic and \emph{connected}
graphs. Since \ca graphs are always connected,
it suffices to require acyclicity.
In the case of \ca graphs, every tree
must be infinite. Take as an example the \ca graph
of the free group over~$E$, for a set of involutive
generators~$E$.
A finite \ca graph can never be fully acyclic,
but finite and $k$-acyclic \ca graphs
can be constructed easily.

\begin{prop}\cite{DawarOttoAPAL09}
 For every finite set~$E$ and every $k\in\N$
 there is a finite, $k$-acyclic \ca graph with
 generator set~$E$. 
\end{prop}

If~$G$ is a tree, then two vertices are always connected
If a graph is $2k+1$-acyclic,
then the subgraphs induced by the $k$-neighbourhoods of all vertices
are $k$-acyclic, i.e.\ all $k$-neighbourhoods look like trees.
This implies that paths of \emph{length up to~$k$} in $2k+1$-acyclic
graphs are unique because two vertices at distance $\leq k$
from each other must share some tree-like $k$-neighbourhood.
These concepts generalise to coset acyclic graphs
in non-trivial ways and will be explored in Section~\ref{sec:shortPaths}.

\subsection{Hypergraphs}\label{sec:hypergraphs}

This section introduces hypergraphs, $\alpha$-acyclicity
and other already known related notions like tree
decompositions.
A hypergraph is a generalisation of a graph in which
an edge can contain any number of vertices.

\begin{definition}\label{dfn:hypergraph}
 A \emph{hypergraph} is a structure $\mathcal{A}=(A,S)$
 with a set of vertices~$A$ and a set of hyperedges
 $S\subseteq \mathcal{P}(A)$.\index{hypergraph}\index{hyperedge}
\end{definition}

With a hypergraph $\mathcal{A}=(A,S)$ we associate
its \emph{Gaifman graph} $G(\mathcal{A})=(A,G(S))$ with
an undirected edge relation~$G(S)$ that links
two vertices $a\neq a'$ if $a,a'\in s$, for
some $s\in S$. An $n$-cycle in a hypergraph is
a cycle of length~$n$ in its Gaifman graph,
and an $n$-path in a hypergraph is
a path of length~$n$ in its Gaifman graph.
The distance $d(X,Y)$ in a hypergraph between
two subsets of vertices~$X$ and~$Y$ is the usual
graph-theoretic distance between~$X$ and~$Y$ in its
Gaifman graph, i.e.\ the minimal length of a path
from~$X$ to~$Y$. A \emph{chord} of 
an $n$-cycle or $n$-path is an edge between vertices
that are not next neighbours along the cycle or path.

There are several, non-equivalent ways to define acyclic hypergraphs.
However, all the different notions of acyclicity
coincide for the usual undirected, loop-free graphs.
The following definition of hypergraph acyclicity is the
classical one from~\cite{Berge}, also known as
$\alpha$-acyclicity in~\cite{BeeriFaginetal};
$n$-acyclicity was introduced in~\cite{Otto12JACM}.

\begin{definition}
 A hypergraph $\mathcal{A}=(A,S)$ is \emph{acyclic} if it is \emph{conformal}
 and \emph{chordal}:
 \begin{enumerate}
  \item conformality requires that every clique in the Gaifman graph~$G(\mathcal{A})$
	is contained in some hyperedge $s\in S$;
  \item chordality requires that every cycle in the Gaifman graph~$G(\mathcal{A})$
        of length greater than~3 has a chord.
 \end{enumerate}

 For $n\geq 3$, $\mathcal{A}=(A,S)$ is \emph{$n$-acyclic} if it is \emph{$n$-conformal}
 and \emph{$n$-chordal}:
 \begin{enumerate}
 \setcounter{enumi}{2}
  \item $n$-conformality requires that every clique in $G(\mathcal{A})$
         up to size~$n$ is contained in some hyperedge $s\in S$;
  \item $n$-chordality requires that every cycle in $G(\mathcal{A})$
        of length greater than~3 and up to~$n$ has a chord.
 \end{enumerate}
\end{definition}

\begin{remark}
If a hypergraph is $n$-acyclic, then every induced
substructure of size up to~$n$ is acyclic~\cite{Otto12JACM}.
\end{remark}

Conformal and chordal hypergraphs are called acyclic
because they are tree-like
in the sense that they are \emph{tree decomposable}.

\begin{definition}
 A hypergraph~$(A,S)$ is \emph{tree decomposable} if it
 admits a tree decomposition $\mathcal{T}=(T,\delta)$:
 $T$ is a tree and $\delta\colon T\to S$ is a map
 such that $\mathrm{image}(\delta)=S$ and, for every
 node $a\in A$, the set $\{v\in T : a\in\delta(v)\}$
 is connected in~$T$.\index{tree decomposition}
\end{definition}
A well-known result from classical hypergraph theory
states that a hypergraph is tree decomposable
if and only if it is acyclic
(see~{\cite{Berge}},~\cite{BeeriFaginetal}).

\section{Acyclicity in \ca graphs and hypergraphs}\label{sec:cycles}

This chapter is concerned with a more general
notion of cycles called \emph{coset cycles};
it was introduced by Otto in~\cite{Otto12JACM}.
Some of the results in this section
can also be found in~\cite{CO17}.


\subsection{Coset acyclicity}\label{sec:cosetCycles}

We can write a labelled cycle of length~$m$ as a finite
sequence $((v_i,e_i))_{i\in\Z_m}$ of pairs from $G\times E$
with $(v_i,v_\ip)\in R_{e_i}$, for all $i\in\Z_m$.
In such an ordinary cycle, every step from~$v_i$ to~$v_\ip$
goes along exactly one edge.
\emph{Coset cycles} allow for steps
that consist of multiple edges at once, or in other words some group
element that is the product of multiple generators from some
subset $\alpha\subseteq E$.
To differentiate ordinary cycles from coset cycles, we use the
following conventions. A \emph{cycle} can both be a finite sequence
of the form $((v_i,e_i))_{i\in\Z_m}$, $e_i\in E$, or $((v_i,\alpha_i))_{i\in\Z_m}$,
$\alpha_i\subseteq E$, where $v_i^{-1} v_\ip\in G_{e_i}$ or
$v_i^{-1} v_\ip\in G_{\alpha_i}$, respectively. A \emph{generator cycle}
is cycle of the form $((v_i,e_i))_{i\in\Z_m}$, where all~$e_i$
are single generators.

\begin{definition}\label{def:OttoCosetCycle}
Let $G$ be a \ca graph with generator set~$E$.
A \emph{coset cycle of length~$m$ in $G$} is a finite sequence
$((v_i,\alpha_i))_{i\in\Z_m}$ with $v_i\in G$ and $\alpha_i\subseteq E$,
for all $i \in \Z_m$, where
$v_i^{-1}v_{i+1}\in G_{\alpha_i}$ and 
\[ 
 [v_{i}]_{\alpha_{i-1}\cap \alpha_i}\cap
 [v_{i+1}]_{\alpha_i\cap\alpha_{i+1}}=\emptyset. 
\]
\end{definition}

\begin{remark}
 For $((v_i,\alpha_i))_{i\in\Z_m}$
 we call
 $[v_{i}]_{\alpha_{i-1}\cap \alpha_i}\cap
 [v_{i+1}]_{\alpha_i\cap\alpha_{i+1}}=\emptyset$ the
 \emph{coset cycle property}.
 It essentially states that every $\alpha_i$-step
 from~$v_i$ to~$v_\ip$ has to count in the sense that it cannot
 be replaced by the previous $\alpha_\im$-step and the subsequent
 $\alpha_\ip$-step. Without this property we would admit ``too many''
 cycles and would not obtain a sensible theory for coset cycles.
\end{remark}

\begin{definition}\label{def:OttoNacyclic}
A \ca graph is \emph{acyclic} if it does not contain a coset
cycle, and \emph{$n$-acyclic} if it does not contain a coset
cycle of length up to~$n$.
\end{definition}

This definition leads to a theory of coset acyclic
\ca graphs that is interesting in itself and has
been shown to be useful for applications in finite
model theory in~\cite{Otto12JACM} and~\cite{CO17}. The exploration
of the structure theory of coset acyclic \ca graphs
is the main topic of this work.
For the remainder of this work, if we speak about
acyclic or $n$-acyclic \ca graphs, we always mean
coset acyclic or coset $n$-acyclic. Acyclicity in the
usual graph-theoretic sense will be indicated specifically.

Coset acyclicity is of further special interest because
every \ca group can be \emph{covered} by an acyclic group
and every finite \ca group can be covered by a finite
$n$-acyclic group, for arbitrary $n$.

\begin{definition}
 A homomorphism $\pi\colon \hat{G}\to G$ is a
 \emph{covering} of~$G$ by~$\hat{G}$ if
 it is surjective and for every $v\in V[\hat{G}]$,
 the restriction of~$\pi$ to the 1-neighbourhood
 of~$v$ is an isomorphism onto the 1-neighbourhood
 of $\pi(v)$.
 If $\pi\colon\hat{G}\to G$ is a covering, we also
 often refer to the structure~$\hat{G}$ as a
 covering of~$G$, or say that~$\hat{G}$ covers~$G$. 
\end{definition}

If~$G$ is a \ca group that is generated by~$E$,
we can construct a covering~$\pi\colon\hat{G}\to G$
and give the function rule of~$\pi$ based on the representation
of a group element~$v$ as a word over~$E$.
However, since an element~$v$ can be
represented by multiple words, the covering must be
compatible with the original group in the following
sense.

\begin{definition}
 Let~$H$ and~$G$ be groups with generator set~$E$.
 $H$ is \emph{compatible} with~$G$ if for all words~$w$ over~$E$ if
 $[w]^G=1$ implies $[w]^H=1.$
\end{definition}

If~$H$ is compatible with~$G$, it is easy to see that~$G$
in fact covers~$H$.

\begin{remark}
 If~$H$ is compatible with~$G$, then
 $ \pi:G\to H, [w]^G\mapsto [w]^H $
 is a well-defined, surjective group homomorphism.
 In particular, $\pi$ is a covering of~$H$ by~$G$.
\end{remark}

Fully acyclic and infinite coverings can be obtained easily
by using the free group over~$E$. Constructing finite, fully acyclic
coverings is out of the question.
But Otto showed in~\cite{Otto12JACM} that it is
possible to construct finite coverings that have an arbitrarily
high degree of acyclicity:

\begin{lemma}\label{le:OttoNacyclic}
 For every finite \ca group~$G$ with finite generator set~$E$
 and every $n\in\N$, there is a finite, $n$-acyclic \ca group~$\hat{G}$
 with generator set~$E$ such that~$G$ is compatible with~$\hat{G}$, and
 $ \pi\colon\hat{G}\to G, [w]^{\hat{G}} \mapsto [w]^G $
 is a covering.
\end{lemma}

Many concepts for graphs that are
acyclic in the usual sense can be generalise to \ca graphs that are
coset acyclic. We establish several close connections
between acyclic \ca graphs and $\alpha$-acyclic hypergraphs,
and argue that acyclic \ca graphs can be considered tree-like
in a more general sense. 
First, we take a closer look at 2-acyclicity
because it provides the backbone for most of the
forthcoming definitions and all further analysis.

\subsection{2-acyclicity}\label{sec:2acyc}

A \ca graph is 2-acyclic if there are no coset cycles of length~2,
i.e.\ if for all vertices $v,u$ and all sets of generators $\alpha,\beta$
with $[v]_\alpha=[u]_\alpha$ and $[v]_\beta=[u]_\beta$:
$[v]_{\alpha\cap\beta}\cap [u]_{\alpha\cap\beta}\neq\emptyset$.
2-acyclicity imposes
a high degree of order in \ca graphs.

\begin{lemma}\label{le:cutChar}
 A \ca graph~$G$ is $2$-acyclic if and only if for all
 $v\in G, \alpha,\beta\subseteq E$
 \[ [v]_\alpha\cap [v]_\beta=[v]_{\alpha\cap\beta}. \]
\end{lemma}

\begin{proof}
 "$\Leftarrow$": If there is a 2-cycle $v,\alpha,u,\beta,v$, then
 $u\in[v]_\alpha\cap[v]_\beta$ and
 $[v]_{\alpha\cap\beta}\cap [u]_{\alpha\cap\beta}=\emptyset$.
 In particular, this means $u\notin [v]_{\alpha\cap\beta}$, which implies
 $[v]_\alpha\cap [v]_\beta\neq[v]_{\alpha\cap\beta}$.
 
 "$\Rightarrow$": Assume there are $v\in G$, $\alpha,\beta\subseteq E$
 such that $[v]_\alpha\cap [v]_\beta\neq[v]_{\alpha\cap\beta}$.
 Since by definition always $[v]_{\alpha\cap\beta}\subseteq[v]_\alpha\cap [v]_\beta$,
 there must be some $u\in([v]_\alpha\cap [v]_\beta)\setminus[v]_{\alpha\cap\beta}$.
 In particular, $u\notin[v]_{\alpha\cap\beta}$ implies
 $[v]_{\alpha\cap\beta}\cap[u]_{\alpha\cap\beta}=\emptyset$.
 Hence $v,\alpha,u,\beta,g$ forms a 2-cycle.
\end{proof}

\begin{example}
 A \ca graph can be of girth~4
 without being even coset 2-acyclic:
 The symmetric group~$S_3$ generated by the transpositions
 $(1,2),(1,3),(2,3)$ has such a \ca graph. Its shortest
 cycle has length~4, but it contains the coset 2-cycle
 $(1),\{(1,2),(2,3)\},(1,3),\{(1,3)\},(1)$.
 This example further illustrates that there is no unique
 minimal connecting subset of generators between two group
 elements; both $\{(1,2),(2,3)\}$ and $\{(1,3)\}$ connect~$(1)$
 and~$(1,3)$, but neither is contained in the other.
 This is not the case in coset 2-acyclic graphs,
 as Lemma~\ref{le:2acycProps} shows. 
\end{example}

The characterisation of 2-acyclicity in Lemma~\ref{le:cutChar} implies
that the intersections of cosets with different subsets of generators in
2-acyclic \ca groups are already far form arbitrary.
As mentioned above, 2-acyclicity provides the backbone of our
further structural analysis. Lemma~\ref{le:2acycProps} shows that in 2-acyclic
groups two elements~$v,u$ are always connected by some unique minimal
set of generators~$\alpha$, i.e.\ $[v]_{\beta}=[u]_{\beta}$ if and only if
$\beta\supseteq\alpha$.
Before we present the lemma, we define the \emph{dual hyperedge}.

\begin{definition}\label{def:dualHyperedge}
In a \ca graph~$G$, define the \emph{dual hyperedge}
induced by an element~$v$ to be the set of cosets
that contain~$v$:
\[
\brck{v} := \{ [v]_\alpha \colon \alpha \subseteq E \}
\]
\end{definition}

\begin{remark}
 In a \ca graph~$G$ for all $v,u\in G$ and all $\alpha\subseteq E$:
 \[
  [v]_\alpha=[u]_\alpha \quad \Leftrightarrow \quad
  v\in [u]_\alpha \quad \Leftrightarrow \quad
  [u]_\alpha \in \brck{v}
 \]
\end{remark}

\begin{lemma}\label{le:2acycProps}
In a $2$-acyclic~\ca group~$G$ with elements $v, v_1,\dots,v_k$ 
and sets of generators $\alpha_1,\ldots,\alpha_k\subseteq E$:
\begin{enumerate}
\item
 For $\beta := \bigcap_{1\leq i\leq k}\alpha_i$:
 $$ v\in\bigcap_{1\leq i\leq k}[v_i]_{\alpha_i}
    \qquad \Rightarrow \qquad
    \bigcap_{1\leq i\leq k}[v_i]_{\alpha_i} = [v]_\beta$$
\item
 The set $\bigcap_{1\leq i\leq k}\brck{v_i}$ 
 has a least element in the sense that there is an $\alpha_0 \subseteq E$
 such that $[v_1]_{\alpha_0} \in \bigcap_{1\leq i\leq k}\brck{v_i}$
 and, for any $\alpha \subseteq E$: 
 $$ [v_i]_{\alpha} \in \bigcap_{1\leq i\leq k}\brck{v_i}
 \qquad \Leftrightarrow \qquad
 \alpha_0 \subseteq \alpha' $$
\end{enumerate}
\end{lemma}

\begin{proof}
\begin{enumerate}
 \item
   Lemma~\ref{le:cutChar} implies
   $\bigcap_{1\leq i\leq k}[v_i]_{\alpha_i} = \bigcap_{1\leq i\leq k}[v]_{\alpha_i} = [v]_\beta$.
 \item
   $2$-acyclicity implies that the collection
   \[
    \{ \alpha \subseteq E \colon [v_1]_\alpha \in \bigcap_{1\leq i\leq k}\brck{v_i}\}
   \]
   is closed under intersections: otherwise there would be $\alpha,\beta\subseteq E$
   with
   $$[v_1]_\alpha,[v_1]_\beta\in\bigcap_{1\leq i\leq k}\brck{v_i}
   \qquad \mbox{and} \qquad
   [v_1]_{\alpha\cap\beta}\notin\bigcap_{1\leq i\leq k}\brck{v_i}.$$
   This implies $[v_1]_{\alpha\cap\beta}\notin\brck{v_j}$, but
   $[v_1]_{\alpha},[v_1]_{\beta}\in\brck{v_j}$,
   for some $1\leq j\leq k$.
   Hence, there would be a 2-cycle $v_1,\alpha,v_j,\beta,v_1$.
\end{enumerate}
\end{proof}

Lemma~\ref{le:2acycProps} justifies the following definition.

\begin{definition}
 In a $2$-acyclic \ca graph we denote the unique \emph{minimal set of
 generators that connects the vertices in a tuple $\vbar$}
 by~$\gen(\vbar)\subseteq E$.
\end{definition}

Intuitively, $\gen(\vbar)$ sets the scale for zooming-in on
the minimal substructure that connects the vertices~$\vbar$.
It behaves in a regular manner.

\begin{lemma}\label{le:addAgent}\label{le:oneLessAgent}
In a $2$-acyclic \ca graph~$G$ for vertices $v,u$ and
  every generator $e\notin \gen(v,u)$:
  \[\gen(v,u\circ e)=\gen(v,u)\cup\{e\}\]
\end{lemma}

\begin{proof}
Set $\alpha:=\gen(v,u)$, and
   let $e\in E\setminus\gen(v,u)$, $u':=u\circ e\neq u$,
   and set $\beta:=\gen(v,u')$.
   The choice of~$u'$ implies an $(\alpha\cup\{e\})$-path from~$v$ to~$u'$.
   Hence, $\beta\subseteq (\alpha\cup\{e\})$ because of 2-acyclicity
   and Lemma~\ref{le:2acycProps}.
   
   Assume $\beta\subsetneq (\alpha\cup\{e\})$.
   First, if $e\notin\beta$, then $\beta\subseteq\alpha$,
   which means there is an $\alpha$-path from~$v$
   to~$u'$ that can be combined with the $\alpha$-path from~$v$ to~$u$
   to an $\alpha$-path from~$u$ to~$u'$. Furthermore,
   $[u]_{\alpha\cap\{e\}}=[u]_\emptyset=\{u\}$ and
   $[u']_{\{e\}\cap\alpha}=[u']_\emptyset=\{u'\}$ since $e\notin\alpha$.
   Together with $u\neq u'$ this implies that $v,\alpha,u',e,u$ forms
   a 2-cycle. Thus, $a\in\beta$ since~$G$ is 2-acyclic.\\
   Second, assume there is some generator $e'\in\alpha$ with $e'\notin\beta$.
   Additionally, $e\in\beta$ and $[u]_e=[u']_e$ imply
   $[u]_\beta=[u']_\beta=[v]_\beta.$
   However, if $\beta\cap\alpha \subsetneq \alpha$, then a $\beta$-path
   from~$v$ to~$u$ contradicts the minimality property of~$\alpha$.
\end{proof}

Lemma~\ref{le:agentsAndClasses} gives us some additional useful insight
into the structure of 2-acyclic \ca graphs.

\begin{lemma}\label{le:agentsAndClasses}
 Let~$G$ be a 2-acyclic \ca graph.
 Then, for all vertices~$v$ and all $\alpha,\beta\subseteq E$,
 $ \beta\subseteq\alpha $ if and only if $[v]_\beta \subseteq [v]_\alpha.$
\end{lemma}

\begin{proof}
 The direction from left to right is, of course,
 true in general.
 
 For the converse direction, let $e\in\beta$, and
 assume $e\notin\alpha$. Since $e\neq 1$,
 the element $v'=v\circ e\in [v]_e$
 is different from~$v$. Additionally,
 $v'\in[v]_e\subseteq[v]_\beta\subseteq[v]_\alpha$
 implies an $\alpha$-path from~$v$ to~$v'$.
 However, this means that $v,\{e\},v',\alpha,v$
 is a coset cycle of length~2 since
 $$ [v]_{\{e\}\cap\alpha} \cap [v']_{\alpha\cap\{e\}}
  = [v]_\emptyset \cap [v']_\emptyset
  = \{v\} \cap \{v'\}
  =\emptyset,
 $$
 which contradicts the assumption of 2-acyclicity.
\end{proof}

Lemma~\ref{le:cosetCut} gives another characterisation
of the coset cycle property in 2-acyclic \ca graphs that
provides a helpful tool in dealing with coset cycles;
it's proof is straightforward.

\begin{lemma}\label{le:cosetCut}
If~$G$ is a 2-acyclic \ca group and $(v_i,\alpha_i)_{i\in\Z_m}$ a finite
sequence with $[v_i]_{\alpha_i}=[v_\ip]_{\alpha_i}$, for all $i\in\Z_m$.
Then for all $i\in\Z_m$
\[
	[v_{i}]_{\alpha_{i-1}\cap \alpha_i}\cap [v_{i+1}]_{\alpha_i\cap \alpha_{i+1}}=
	[v_{i-1}]_{\alpha_{i-1}} \cap [v_{i}]_{\alpha_{i}} \cap [v_{i+1}]_{\alpha_{i+1}}.
\]
\end{lemma}


\subsection{Dual hypergraphs}\label{sec:dualHypgergraphs}

In this section, we define for every \ca graph~$G$ an
associated structure~$d(G)$, the \emph{dual hypergraph
of~$G$}, and present the first connections between
coset acyclicity for \ca graphs and $\alpha$-acyclicity
for their dual hypergraphs.

\begin{definition}\label{def:dualhypergraphdef}
Let $G=(V,(R_\alpha)_{\alpha\in E})$ 
be a \ca graph, and define the equivalence relation
$R_\alpha:=\mathsf{TC}(\bigcup_{e \in \alpha} R_e)$,
for all $\alpha\subseteq E$ ($\mathsf{TC}$ denotes the transitive closure).
The \emph{dual hypergraph}
of~$G$ is the vertex-coloured hypergraph
 \begin{align*}\textstyle
 d(G) &:= (d(V),S, (Q_\alpha)_{\alpha\subseteq E}) \mbox{ where}
\\
d(V)&:= \dot{\bigcup}_{\alpha\subseteq E} Q_\alpha \;\;\mbox{ for }
      Q_\alpha :=  V/R_\alpha, \\
    S &:= \{\brck{v} \subseteq d(V): v\in V \}.
 \end{align*}
\end{definition}

As the name suggests, everything in the dual hypergraph is flipped.
The vertices of~$G$ are the hyperedges of~$d(G)$,
the $\alpha$-cosets of~$G$ are the $\alpha$-coloured
vertices of~$d(G)$.
Furthermore, Lemma~\ref{le:2acycProps} implies that every intersection
between hyperedges can be described by the unique set of generators $\gen(\vbar)$.
This means, for every $v\in\vbar$ and every $\alpha\subseteq E$:
$$ [v]_\alpha \in \bigcap_{v\in\vbar} \brck{v}
   \qquad \Leftrightarrow \qquad
   \alpha \supseteq \gen(\vbar)
$$

The notions of acyclicity for \ca graphs and hypergraph
acyclicity are directly connected.
Otto showed that the dual hypergraph~$d(G)$ is $n$-acyclic
if~$G$ is coset $n$-acyclic, and we show the other direction
for 2-acyclic~$G$.

\begin{lemma}\label{le:acycCayleyHyp}\cite{Otto12JACM}
 For~$n\geq 3$, if $G$ is an $n$-acyclic \ca graph,
 then $d(G)$ is an $n$-acyclic hypergraph.
\end{lemma}

\begin{lemma}
 Let~$G$ be a $2$-acyclic \ca graph.
 For~$n\geq 3$, if $d(G)$ is an $n$-acyclic hypergraph,
 then~$G$ is $n$-acyclic.
\end{lemma}

\begin{proof}
 Let $((v_i,\alpha_i))_{i\in\Z_m}$ be a coset cycle
 of minimal length in~$G$. We need to show that $m>n$.
 The cycle $((v_i,\alpha_i))_{i\in\Z_m}$ in~$G$ induces
 an associated cycle $(([v_i]_{\alpha_i},\brck{v_\ip}))_{i\in\Z_m}$
 in the dual hypergraph~$d(G)$ because
 $[v_i]_{\alpha_i}\in\brck{v_\ip}$ since $(v_i,v_\ip)\in R_{\alpha_i}$.
 If we show that this cycle is chordless, then $n$-acyclicity
 of~$d(G)$ implies $m>n$.
 
 The length of $((v_i,\alpha_i))_{i\in\Z_m}$ is at least be~$3$
 because~$G$ is 2-acyclic. If it is~$3$, then the induced
 cycle $(([v_i]_{\alpha_i},\brck{v_\ip}))_{i\in\Z_m}$ must
 be contained in some hyperedge~$\brck{v}$ because~$d(G)$ is,
 in particular, 3-conformal. However, the definition of~$d(G)$
 and 2-acyclicity of~$G$ together with Lemma~\ref{le:cosetCut} imply
 \begin{align*}
  &[v_1]_{\alpha_1}\cap[v_2]_{\alpha_2}\cap[v_3]_{\alpha_3}\in\brck{v} \\
  \Rightarrow\quad
  &v\in[v_1]_{\alpha_1}\cap[v_2]_{\alpha_2}\cap[v_3]_{\alpha_3}
   =[v_1]_{\alpha_3\cap\alpha_1}\cap[v_2]_{\alpha_1\cap\alpha_2};
 \end{align*}
 this violates the coset cycle property
 $[v_1]_{\alpha_3\cap\alpha_1}\cap[v_2]_{\alpha_1\cap\alpha_2}=\emptyset$.
 Hence, $m$ must be at least 4.
 
 Now, assume that the cycle $(([v_i]_{\alpha_i},\brck{v_\ip}))_{i\in\Z_m}$
 has a chord, i.e.\ there is
 some hyperedge~$\brck{u}$ and there are $1\leq i,j\leq m$
 with $j>i+1$ such that $[v_i]_{\alpha_i},[v_j]_{\alpha_j}\in\brck{u}$.
 First, we choose~$\brck{u}$ such that the distance between~$[v_i]_{\alpha_i}$
 and~$[v_j]_{\alpha_j}$ on the cycle is \emph{minimal}, i.e.\ there are
 no other vertices on the cycle that are connected by a chord and have
 a shorter distance on the cycle than~$[v_i]_{\alpha_i}$ and~$[v_j]_{\alpha_j}$.
 Then
 $$[u]_{\alpha_i},\brck{v_\ip},[v_\ip]_{\alpha_\ip},\dots,
  [v_\jm]_{\alpha_\jm},\brck{v_j},[u]_{\alpha_j},\brck{u},[u]_{\alpha_i}$$
 is a cycle in $d(G)$ since $[u]_{\alpha_i}=[v_i]_{\alpha_i}$ and
 $[u]_{\alpha_j}=[v_j]_{\alpha_j}$. This cycle in the dual hypergraph
 induces a cycle 
 $$ u,\alpha_i,v_\ip,\alpha_\ip,\dots,\alpha_\jm,v_j,\alpha_j,u$$
 in~$G$ of length shorter than~$m$. If we can show that this cycle
 is also a \emph{coset} cycle, then the chord~$\brck{u}$ could
 not exists because it would contradict that we chose
 $((v_i,\alpha_i))_{i\in\Z_m}$ as a coset cycle of minimal length.
 
 We need to check the coset property at~$u$,
 i.e.\ $[u]_{\alpha_j\cap\alpha_i}\cap[v_\ip]_{\alpha_i\cap\alpha_\ip}=\emptyset$
 and $[v_j]_{\alpha_\jm\cap\alpha_j}\cap[u]_{\alpha_j\cap\alpha_i}=\emptyset$.
 Assume there is some $w\in[u]_{\alpha_j\cap\alpha_i}\cap[v_\ip]_{\alpha_i\cap\alpha_\ip}$.
 2-acyclicity of~$G$ and Lemma~\ref{le:cosetCut} imply
 $w \in [u]_{\alpha_j}\cap[v_i]_{\alpha_i}\cap[v_\ip]_{\alpha_\ip}
       =[v_j]_{\alpha_j}\cap[v_i]_{\alpha_i}\cap[v_\ip]_{\alpha_\ip}$.
 We assumed~$\brck{u}$ to be such that the distance between~$[v_i]_{\alpha_i}$
 and~$[v_j]_{\alpha_j}$ on the cycle is minimal, hence $j>i+2$ cannot be the case
 because $[v_\ip]_{\alpha_\ip},[v_j]_{\alpha_j}\in\brck{w}$ have shorter
 distance. This leaves $j=i+2$, which implies
 \[
  \emptyset \neq [v_i]_{\alpha_i}\cap[v_\ip]_{\alpha_\ip}\cap[v_{i+2}]_{\alpha_{i+2}}
                =[v_\ip]_{\alpha_i\cap\alpha_\ip}\cap[v_{i+2}]_{\alpha_\ip\cap\alpha_{i+2}}.
 \]
 But this contradicts the coset property of the given coset cycle. Showing
 $[v_j]_{\alpha_\jm\cap\alpha_j}\cap[u]_{\alpha_j\cap\alpha_i}=\emptyset$
 works analogously.
 
 Thus, we found a coset cycle that is shorter than~$m$.
 This contradicts the choice of $((v_i,\alpha_i))_{i\in\Z_m}$ as a
 coset cycle of minimal length in~$G$.
 This means that $(([v_i]_{\alpha_i},\brck{v_\ip}))_{i\in\Z_m}$
 must be chordless, which implies $m>n$ by $n$-acyclicity of~$d(G)$.
\end{proof}

Thus, the previous lemmas show that an acyclic \ca graph is tree-like
in the sense that its dual hypergraph is tree-decomposable.

\section{Analysis of paths and distances}\label{sec:paths}

Coset cycles generalise the graph-theoretic notion
of a cycle for \ca graphs.
\emph{Coset paths} generalise the graph-theoretic notion of a path
in the same way.
These coset paths and their behaviour in $n$-acyclic \ca graphs
are the subject of this chapter.

Many of the various definitions and notions that we
introduce from now on only make sense in 2-acyclic
\ca graphs, because they are based on
the set $\gen(\vbar)$.
Therefore, and because every \ca graph has
a 2-acyclic covering, we make the following assumption
for the remainder of this section.

\begin{pprov}
 Every \ca graph is assumed to be 2-acyclic.
\end{pprov}

\begin{definition}[Coset path]\label{def:cosetPath}
 Let~$G$ be a \ca graph.
 A \emph{coset path of length~$\ell\geq 1$}
 is a labelled path
 $v_1,\alpha_1,v_2,\alpha_2,\dots,\alpha_\ell,v_{\ell+1}$
 such that, for $1\leq i \leq \ell$,
 \[
  [v_i]_{\alpha_{i-1}\cap \alpha_i} \cap
  [v_{i+1}]_{\alpha_{i}\cap\alpha_{i+1}}=\emptyset,
 \]
 with $\alpha_0=\alpha_{\ell+1}=\emptyset$. 
 A coset path $v_1,\alpha_1,\dots,\alpha_{\ell},v_{\ell+1}$
 of length $\ell\geq 2$ is \emph{non-trivial} if,
 for $\alpha=\gen(v_1,v_\ellp)$, for all $1\leq i\leq\ell$,
 \[ [v_1]_\alpha \nsubseteq [v_i]_{\alpha_i}. \]
 A coset path $v_1,\alpha_1,\dots,\alpha_{\ell},v_{\ell+1}$
 of length $\ell\geq 2$ is an \emph{inner} path if,
 for $\alpha=\gen(v_1,v_\ellp)$, for all $1\leq i\leq \ell$,
 \[ [v_i]_{\alpha_i}\subsetneq [v_1]_\alpha. \]
 A non-trivial coset path from~$v$ to~$u \not=v$ is
 \emph{minimal} if there is no shorter non-trivial coset
 path from~$v$ to~$u$.
\end{definition}

\begin{remark}
 Non-trivial and inner coset paths are only
 well-defined in 2-acyclic graphs.
\end{remark}

\begin{obs}
 Inner coset paths are non-trivial.
\end{obs}

In other words, a coset path is a path that links
two consecutive vertices not via a single edge or
generator, but via a coset in a way that respects
the coset property of coset cycles in every step.
An analogue of Lemma~\ref{le:cosetCut} is also true
for coset paths.

\begin{lemma}\label{le:cosetCutPath}
If~$G$ is a \ca graph and $v_1,\alpha_1,v_2,\dots,v_{\ell},\alpha_\ell,v_{\ell+1}$
a path, then, for all $2 \leq i\leq\ell$,
\[
	[v_{i}]_{\alpha_{i-1}\cap \alpha_i}\cap [v_{i+1}]_{\alpha_i\cap \alpha_{i+1}}=
	[v_{i-1}]_{\alpha_{i-1}} \cap [v_{i}]_{\alpha_{i}} \cap [v_{i+1}]_{\alpha_{i+1}},
\]
with $\alpha_{\ell+1}=\emptyset$.
\end{lemma}


The following sections develop a theory of coset
paths in $n$-acyclic \ca graphs.

\subsection{Short coset paths}\label{sec:shortPaths}

If a \ca graph is $2k+1$-acyclic in the usual sense,
then every $k$-neighbourhood $N^k(v)$ induces a substructure
that is a tree. This entails that two vertices that
have a distance of at most~$k$ are connected by a unique
path of length at most~$k$. This concept generalises to coset acyclic
\ca graphs w.r.t.\ coset paths.

In an acyclic \ca graph, two distinct vertices~$v$ and~$u$
are always uniquely connected by a coset path of the form
$v_1,\{e_1\},\dots,\{e_\ell\},v_\ellp$ where
all the sets of generators are singletons.
But there might be a myriad of different recombinations
of sets of these generators that pass as proper coset paths.
However, all these paths overlap in some sense, and if the \ca
graph is $2n$-acyclic all paths of length up to~$n$ overlap
in this way. This is the content of the zipper lemma
(Lemma~\ref{le:zipper}), the central result of this section.
Let us make precise what we mean by \emph{short}
coset paths.

\begin{definition}
 Let~$G$ be a~\ca graph that is $2n$-acyclic. We call
 a coset path \emph{short} if its length is $\leq n$.
\end{definition}

Often we do not make it explicit to what degree a \ca graph is
acyclic. Instead, we write that a \ca graph~$G$ is
\emph{sufficiently} acyclic, i.e.\ there is some
$n\in\N$ such that~$G$ is $n$-acyclic and
all the arguments go through.

Essentially, the zipper lemma states that in a sufficiently
acyclic \ca graph two short coset paths that both start at
the same vertex~$v$ and end at the same vertex~$u$ overlap non-trivially
at \emph{both} ends. Thus, multiple applications of the
zipper lemma imply that two short coset paths of this kind
behave like a zipper that can be closed from both ends.
Furthermore, the zipper lemma
implies that, for all pairs of vertices~$(v,u)$,
there is a unique minimal set of generators~$\alpha_0$
such that $\alpha_0\subseteq\alpha_1$, for all
\emph{short} coset paths $v,\alpha_1,\dots,\alpha_\ell,u$.
This set~$\alpha_0$ can be interpreted as
the direction one \emph{has to} take if one wants to
move from~$v$ to~$u$ on a short coset path.

\medskip
In order to prove the zipper lemma, we begin with
considering short coset paths
$v_1,\alpha_1,v_2,\dots,v_\ell,\alpha_\ell,v_1$
that start and end at the same vertex~$v_1$.
Such a path may differ from a coset
cycle regarding the overlaps at the ends. If
$v_1,\alpha_1,v_2,\dots,v_\ell,\alpha_\ell,v_1$ is just a \emph{path},
we can by definition only assume
\[
 [v_1]_{\emptyset\cap\alpha_1}\cap[v_2]_{\alpha_1\cap\alpha_2}=\emptyset
 \quad\mbox{and}\quad
 [v_\ell]_{\alpha_\ellm\cap\alpha_\ell}\cap[v_1]_{\alpha_\ell\cap\emptyset}=\emptyset,
\]
i.e.\ $v_1\notin [v_2]_{\alpha_1\cap\alpha_2}$ and
$v_1\notin [v_\ell]_{\alpha_\ellm\cap\alpha_\ell}$,
but not that it is a complete coset \emph{cycle}, i.e.\ that also
\[
 [v_1]_{\alpha_\ell\cap\alpha_1}\cap[v_2]_{\alpha_1\cap\alpha_2}=\emptyset
 \quad\mbox{and}\quad
 [v_\ell]_{\alpha_\ellm\cap\alpha_\ell}\cap[v_1]_{\alpha_\ell\cap\alpha_1}=\emptyset.
\]
Hence, these cyclic coset paths are not directly ruled out
by acyclicity but by the following lemma.

\begin{lemma}\label{le:cyclicZipper}
 Let~$v$ be a vertex in a \ca graph~$G$. If~$G$ is $n$-acyclic, 
 then there is no coset path of length up to~$n$ that starts
 and ends at~$v$.
\end{lemma}

\begin{proof}
 The claim is shown by induction on the length~$\ell$
 of the coset path, for $1\leq\ell\leq n$.
 
 For $\ell=1$, Definition~\ref{def:cosetPath} rules out
 coset loops $v,\alpha,v$ because it implies
 \[
  \emptyset=[v]_{\emptyset\cap\alpha} \cap [v]_{\alpha\cap\emptyset}=\{v\}.
 \]
 
 For $\ell=2$, coset paths $v_1,\alpha_1,v_2,\alpha_2,v_1$
 with $v_1\notin[v_2]_{\alpha_1\cap\alpha_2}$ are ruled
 out because 2-acyclicity implies
 $$[v_1]_{\alpha_1\cap\alpha_2}=[v_2]_{\alpha_1\cap\alpha_2},$$
 leading to the contradiction $v_1\notin[v_1]_{\alpha_1\cap\alpha_2}$.
 
 For $2<\ell\leq n$, assume there are no coset paths
 of length up to~$\ell-1$ from any vertex back to itself.
 Consider a coset path
 \[
  v_1,\alpha_1,v_2,\dots,v_\ell,\alpha_\ell,v_\ellp
 \]
 of length~$\ell$ with $v_1=v_\ellp$. That~$G$ is $n$-acyclic implies
 \[
  [v_1]_{\alpha_\ell\cap\alpha_1}\cap[v_2]_{\alpha_1\cap\alpha_2}\neq\emptyset
  \quad \mbox{or} \quad
  [v_\ell]_{\alpha_\ellm\cap\alpha_\ell}\cap[v_1]_{\alpha\ell\cap\alpha_1}\neq\emptyset.
 \]
 W.l.o.g.\ we assume there is some
 $u\in[v_1]_{\alpha_\ell\cap\alpha_1}\cap[v_2]_{\alpha_1\cap\alpha_2}$.
 If $u\notin[v_\ell]_{\alpha_\ellm\cap\alpha_\ell}$, then
 $$u,\alpha_2,v_3,\alpha_3,v_4,\dots,v_\ell,\alpha_\ell,u$$
 is a coset path of length $\ell-1$ from~$u$ to itself. Otherwise,
 $$u,\alpha_2,v_3,\alpha_3,v_4,\dots,v_\ellm,\alpha_\ellm,u$$
 is a coset path of length $\ell-2$ from~$u$ to itself.
 In both cases, such a coset path cannot exist according
 to the induction hypothesis.
\end{proof}

The proof of Lemma~\ref{le:cyclicZipper} shows that a short
cyclic path cannot exist in a sufficiently acyclic graph because
it would collapse onto itself.
The zipper lemma follows easily from this.

\begin{lemma}[Zipper lemma]\label{le:zipper}
Let~$G$ be a $2n$-acyclic \ca graph, $v,u\in G$, and
\[
 v,\alpha_1,t_2,\alpha_2,t_3,\dots,t_\ell,\alpha_\ell,u
 \quad\mbox{and}\quad
 v,\beta_1,r_2,\beta_2,r_3,\dots,r_k,\beta_k,u
\]
be two coset paths from~$v$ to~$u$ of length up to~$n$.
Then
\begin{enumerate}
\item
$[v]_{\beta_1\cap \alpha_1}\cap[t_2]_{\alpha_1\cap\alpha_2}\neq\emptyset$
\; or \;
$[v]_{\alpha_1\cap\beta_1} \cap [r_2]_{\beta_1\cap \beta_2} \neq\emptyset$;
\item
$[u]_{\beta_k \cap \alpha_{\ell}} \cap [t_{\ell}]_{\alpha_\ell\cap\alpha_{\ell-1}}
\neq\emptyset$
\; or \;
$[u]_{\alpha_\ell\cap \beta_k}\cap[r_{k}]_{\beta_k\cap \beta_{k-1}}\neq\emptyset$.

\end{enumerate}
\end{lemma}

\begin{proof}
 Both paths are short and share the start vertex~$v$
 and the end vertex~$u$. Both paths fulfil the coset
 cycle property at every link between~$v$ and~$u$ by
 definition. However, the assumptions do not tell us
 exactly what the situation looks like at~$v$ and~$u$,
 the places where the paths overlap. The zipper lemma
 claims that there is an overlap that violates the
 coset cycle property at \emph{both} ends.
 
 Since~$G$ is $2n$-acyclic we know that there must
 be an overlap at one of the ends, i.e.\ 
 \begin{itemize}
  \item
    $[v]_{\beta_1\cap \alpha_1}\cap[t_2]_{\alpha_1\cap\alpha_2}\neq\emptyset$, or
  \item
    $[v]_{\alpha_1\cap\beta_1} \cap [r_2]_{\beta_1\cap \beta_2}\neq\emptyset$, or
  \item
    $[u]_{\beta_k \cap \alpha_{\ell}} \cap [t_{\ell}]_{\alpha_\ell\cap\alpha_{\ell-1}}\neq\emptyset$, or
  \item
    $[u]_{\alpha_\ell\cap \beta_k}\cap[r_{k}]_{\beta_k\cap \beta_{k-1}}\neq\emptyset$
 \end{itemize}
 occurs because otherwise the two coset paths would form
 a coset cycle of length up to~$2n$; w.l.o.g.\ assume
 $[v]_{\beta_1\cap \alpha_1}\cap[t_2]_{\alpha_1\cap\alpha_2}\neq\emptyset$.
 If we now assume that there is no overlap at~$u$, i.e.\ 
 \[
  [u]_{\beta_k \cap \alpha_{\ell}} \cap [t_{\ell}]_{\alpha_\ell\cap\alpha_{\ell-1}}=\emptyset
  \quad\mbox{and}\quad
  [u]_{\alpha_\ell\cap \beta_k}\cap[r_{k}]_{\beta_k\cap \beta_{k-1}}=\emptyset,
 \]
 then there would be a cyclic coset path of length up to~$2n$
 from~$v$ to~$v$, contradicting Lemma~\ref{le:cyclicZipper}. 
\end{proof}

The zipper lemma states that two short coset paths that start and end
at the same vertices can be considered two recombinations
of the constituents of a common core path.
Short coset paths in acyclic \ca graphs are unique
in the sense that the zipper lemma applies to them.
Thus, $n$-acyclic \ca graphs can be
considered locally tree-like.
The zipper lemma has several important consequences.

\begin{corollary}\label{cor:edgeCut}
 Let~$G$ be a $2n$-acyclic \ca graph, $v,u\in G$.
 If there are two short coset paths
 \[
  v,\alpha_1,t_2,\alpha_2,t_3,\dots,t_\ell,\alpha_\ell,u
  \quad\mbox{and}\quad
  v,\beta_1,r_2,\beta_2,r_3,\dots,r_k,\beta_k,u
 \]
 from~$v$ to~$u$ with $\ell,k\leq n$, then there is a short coset paths
 from~$v$ to~$u$ that starts with an $(\alpha_1\cap\beta_1)$-edge.
\end{corollary}

\begin{proof}
 W.l.o.g.\ we can assume that there
 is some $v_2\in[v]_{\beta_1\cap\alpha_1}\cap[t_2]_{\alpha_1\cap\alpha_2}$
 by Lemma~\ref{le:zipper}. First, the choice of~$v_2$ and the coset property
 of the original path imply
 \[
  [v_2]_{\alpha_1\cap\alpha_2}\cap[t_3]_{\alpha_2\cap\alpha_3}=
  [t_2]_{\alpha_1\cap\alpha_2}\cap[t_3]_{\alpha_2\cap\alpha_3}=\emptyset.
 \]
 Second,
 \[
  v\notin[t_2]_{\alpha_1\cap\alpha_2}=[v_2]_{\alpha_1\cap\alpha_2}
  \supseteq[t_2]_{\alpha_1\cap\alpha_2\cap\beta_1}
 \]
 implies
 \[
  [v]_{\emptyset\cap(\alpha_1\cap\beta_1)}\cap[v_2]_{(\alpha_1\cap\beta_1)\cap\alpha_2}=\emptyset.
 \]
 Thus, $v,(\alpha_1\cap\beta_1),v_2,\alpha_2,t_3,\dots,t_\ell,\alpha_\ell,u$
 is a short coset path.
\end{proof}

Let~$G$ be a 2-acyclic \ca graph and $v,u\in G$. Based on
Corollary~\ref{cor:edgeCut} we define the unique minimal set
of generators~$\sh(v,u)\subseteq E$.

\begin{definition}\label{def:short}
A set of generators~$\alpha$ is a \emph{first generator set for $(v,u)$}
if there is a short coset path from~$v$ to~$u$ that starts with an
$\alpha$-edge. The \emph{minimal first generator set for $(v,u)$} $\sh(v,u)$
is the intersection of all first generator sets:
\[
 \sh(v,u) := \bigcap \{\alpha\subseteq E :
 \mbox{ $\alpha$ is a first generator set for $(v,u)$}\}
\]
\end{definition}

The unique set $\sh(v,u)$ is well-defined because the intersection
of two first generator sets is again a first generator set by
Corollary~\ref{cor:edgeCut}. In general, $\sh(v,u)\neq\sh(u,v)$
but
\[
 \sh(v,u),\sh(u,v)\subseteq\gen(v,u)=\gen(u,v)
\]
because $\gen(v,u)$ is a first generator set for~$(v,u)$ and~$(u,v)$.
The set $\sh(v,u)$ gives us another perspective
on the uniqueness of short coset paths. If one wants to move from one vertex
to another on a short coset path, then there might be many possibilities
but just one single ``direction'' to start with.

Furthermore, the zipper lemma implies that all short
coset paths of length $\geq 2$ can be assumed
to be inner paths.

\begin{corollary}\label{cor:shortIsNontrivial}
 Let~$G$ be a $2n$-acyclic \ca graph, $2\leq\ell\leq n$,
 \[
  v_1,\alpha_1,v_2,\alpha_2,v_3,\dots,v_\ell,\alpha_\ell,v_\ellp
 \]
 be a coset path and $\alpha\supseteq\gen(v_1,v_\ellp)$.
 Then $\alpha_i\nsupseteq\alpha$,
 for $1\leq i\leq\ell$, and there are
 $v_i'\in[v_i]_{\alpha_{i-1}\cap\alpha_i}$, for $1< i\leq\ell$,
 such that
 \[
  v_1,(\alpha_1\cap\alpha),v_2',(\alpha_2\cap\alpha),v_3',\dots,v_\ell',(\alpha_\ell\cap\alpha),v_\ellp
 \]
 is an inner coset path.
\end{corollary}

\begin{proof}
 First, $\alpha_1 \supseteq \alpha$ cannot be the case:
 if $\ell=2$, then $v_1,\alpha_1,v_2,\alpha_2,v_3$ would not be
 a coset path since $v_3\in[v_2]_{\alpha_1\cap\alpha_2}$,
 and $\ell>2$ would imply a short cyclic coset path from~$v_\ellp$
 to itself, contradicting Lemma~\ref{le:cyclicZipper}.
 Hence, in both cases $\alpha_1 \nsupseteq \alpha$, and with that
 $\alpha_1\cap\alpha \subsetneq \alpha$ which implies
 $[v_1]_{\alpha_1\cap\alpha} \subsetneq [v_1]_\alpha.$
 
 Second, analogously to the proof of Corollary~\ref{cor:edgeCut}
 one can show that there is some $v_2'\in[v_2]_{\alpha_{1}\cap\alpha_2}$
 such that
 \[
  v_1,(\alpha_1\cap\alpha),v_2',\alpha_2,v_3,\dots,v_\ell,\alpha_\ell,v_\ellp
 \]
 is a coset path because $v_1,\alpha,v_\ellp$ is also a short
 coset path from~$v_1$ to~$v_\ellp$. Applying the same argument iteratively
 to the paths $v_i',\alpha_i,v_{i+1},\dots,v_\ell,\alpha_\ell,v_\ellp$
 and $v_i',\alpha,v_\ellp$, for $2\leq i\leq\ell$,
 shows $\alpha_i\nsupseteq\alpha$ and yields the desired vertices.
\end{proof}

Corollary~\ref{cor:shortIsNontrivial} illustrates the special
role of the subgraph induced by~$[v]_{\gen(v,u)}$:
all short coset paths between~$v$ and~$u$ essentially
move within~$[v]_{\gen(v,u)}$.
Conversely, if a coset path has a link that is disjoint
from~$[v]_{\gen(v,u)}$, then it must be long.

\begin{corollary}\label{cor:longOnTheOutside}
 Let~$G$ be a $2n$-acyclic \ca graph.
 If $v_1,\alpha_1,\dots,\alpha_\ell,v_\ellp$
 is a coset path with
 \[ [v_1]_{\gen(v_1,v_\ellp)}\cap[v_i]_{\alpha_{i-1}\cap\alpha_i}=\emptyset, \]
 for some $2\leq i\leq\ell$, then $\ell>n$.
\end{corollary}

\subsection{Distance in \ca graphs}\label{sec:distance}

In a \ca graph, every pair of vertices~$v$, $u$ is connected
by the coset path $v,E,u$ of length~1.
This makes the definition of a sensible measure of distance
w.r.t.\ coset paths non-obvious.
However, we find a solution with the help of 2-acyclicity
and its implications.
Using 2-acyclicity and the set~$\gen(v,u)$, for vertices~$v,u$,
we defined non-trivial coset paths, the paths that remain if one forbids
all cosets that connect~$v$ and~$u$ in one step.
This leads us to a non-trivial notion
of distance in 2-acyclic \ca graphs.

\begin{definition}[Distance in \ca graphs]\label{def:distance}
 Let~$G$ be a \ca graph.
 The \emph{distance} $d(v,u)$ between two vertices~$v\neq u$
 is defined as the length of a minimal
 non-trivial coset path from~$v$ to~$u$.
\end{definition}

\begin{remark}
 Definition~\ref{def:distance} does not allow for $d(v,u)=1$.
 This might seem peculiar compared to other distance measures.
 However, the measure $d(v,u)$ is precisely designed to
 capture the length of the \emph{non-trivial} coset path connections
 between two vertices, and their length is always at least~2.
\end{remark}

In the previous section, we showed that in sufficiently
acyclic structures all short coset paths can be considered
inner paths. This has implications for the distance.
If we want to know if the distance between~$v$ and~$u$
is long, it suffices to look at the inner paths within
the substructure induced by~$[v]_{\gen(v,u)}$.

\begin{lemma}\label{le:noShortInner}
 Let $m\in\N$,~$G$ be a sufficiently acyclic
 \ca graph and $v,u$ two vertices.
 If there are no inner coset paths
 from~$v$ to~$u$ of length $\leq m$,
 then $d(v,u)>m.$
\end{lemma}

\begin{proof}
 Let $\ell\leq m$, and assume there is a non-trivial coset path
 $v_1,\alpha_1,\dots,\alpha_\ell,v_\ellp$ of length~$\ell$
 from $v=v_1$ to $u=v_\ellp$.
 First, any non-trivial coset path has at least length~2.
 Second, we can assume that the path is an inner coset path
 by Lemma~\ref{cor:shortIsNontrivial} since~$G$
 is sufficiently acyclic.
 This contradicts our assumption. Thus, $d(v,u)>m$.
\end{proof}

The original motivation for this distance stems from~\cite{CO17}.
The central problem there is to play \EF games on \ca graphs with
their complex overlapping edge patterns w.r.t.\ cosets.
To win an \EF game, one must be able to control distances
between multiple vertices of a structure. In the case of \ca graphs one needs
to find a suitable measure of distance first. The one
from Definition~\ref{def:distance} suffices.

Furthermore, this distance for \ca graphs closely corresponds to
a very natural distance in their dual hypergraphs. In dual hypergraphs,
the two hyperedges~$\brck{v}$ and~$\brck{u}$,
for vertices~$v$ and~$u$, always intersect. This intersection is
exactly the set of $\alpha$-cosets, for $\alpha\subseteq E$, that contain
both~$v$ and~$u$. At first glance,
the distance between two hyperedges seems always trivially~0.
But we obtain a meaningful measure of distance in dual hypergraphs
between~$\brck{v}$ and~$\brck{u}$ if we cut out the intersection
$\brck{v}\cap\brck{u}$ and consider the remaining
paths in the Gaifman graph. Essentially, we look for the non-trivial
paths of minimal length between~$\brck{v}$ and~$\brck{u}$.
\begin{definition}\label{def:distanceHypergraph}
 Let~$G$ be a \ca graph, $v,u\in G$ and $t=\brck{v}\cap\brck{u}$.
 The distance $d(\brck{v},\brck{u})$ between the hyperedges~$\brck{v}$
 and~$\brck{u}$ in the dual hypergraph~$d(G)$ is the usual
 graph-theoretic distance in the Gaifman graph of
 $d(G)\upharpoonright (d(V[G])\setminus t)$ between~$\brck{v}\setminus t$
 and~$\brck{u}\setminus t$.
\end{definition}

It is the main result of this section that this measure of
distance for dual hypergraphs corresponds exactly to the distance
defined in~\ref{def:distance} for \ca graphs if certain acyclicity
conditions are met.
If~$G$ is a 2-acyclic \ca graph
and $v\neq u$ are vertices, then
$d(v,u)=d(\brck{v},\brck{u})+1.$
However, we will prove a more general statement that
has a wider range of graph and model-theoretic applications.
In order to obtain a meaningful notion of distance,
we followed the same idea both in \ca graphs and their
dual hypergraphs: cut out the trivial connections, or more, and look
at what remains.

Let~$G$ be a \ca graph and $v\neq u$ vertices. The intersection of
the dual hyperedges $t=\brck{v}\cap\brck{u}$ is always
a non-empty set of cosets. If~$G$ is 2-acyclic, then~$t$ is generated by
the unique set $\gen(v,u)$ (cf.\ Lemma~\ref{le:2acycProps}), i.e.
$$t=\{ [v]_\beta : \beta\supseteq\gen(v,u) \}=\{ [u]_\beta : \beta\supseteq\gen(v,u) \}.$$
We can further generalise the distance measure
$d(\brck{v},\brck{u})$ if we do not forbid
$\brck{v}\cap\brck{u}$, but a more general
set of cosets that has the same structure
as $\brck{v}\cap\brck{u}$. If such a set is a superset of
$\brck{v}\cap\brck{u}$, we arrive at a more general measure
of distance that still has a correspondent in \ca graphs.
Before we formally define these distances, we introduce
some notation to describe the forbidden sets.

\begin{definition}
 For a 2-acyclic \ca graph~$G=(V,(R_e)_{e\in E})$
 with the dual hypergraph~$d(G)=(d(V),S,(Q_\alpha)_{\alpha\subseteq E})$,
 we define the following mapping: 
 \begin{align*}
  \rho^G\colon &V\times\mathcal{P}(E)\to \mathcal{P}(d(V)), \\
               &(v,\gamma) \mapsto \{ [v]_\beta : \beta \supseteq \gamma \}
 \end{align*}
 If it is clear from the context,
 we drop the superscript~$G$ and just write~$\rho$ instead.
\end{definition}

The following lemma characterises the relationship
of the sets $\brck{v}\cap\brck{u}$ and
$\rho(v,\gamma)$ in~$d(G)$ in terms of
$\gen(v,u)$ and $\gamma$.
We can observe the usual duality in the transition from
\ca graphs to their dual hypergraphs.

\begin{lemma}\label{le:tProp}
 Let~$G$ be a 2-acyclic \ca graph, $v,u$ two vertices
 and~$\gamma\subseteq E$ a set of generators, then
 $ \brck{v}\cap\brck{u} \subseteq \rho(v,\gamma)$
 if and only if  $\gamma \subseteq \gen(v,u).$
\end{lemma}

\begin{proof}
 Put $\alpha:=\gen(v,u)$. From right to left: assume $\gamma \subseteq \alpha$.
 Together with 2-acyclicity this implies
 \[
  \brck{v}\cap\brck{u} = \{ [v]_\beta : \beta\supseteq\alpha \} \subseteq
  \{ [v]_\beta : \beta \supseteq \gamma \} = \rho(v,\gamma).
 \]
 From left to right: assume $\brck{v}\cap\brck{u} \subseteq \rho(v,\gamma)$.
 As before, $\brck{v}\cap\brck{u} = \{ [v]_\beta : \beta\supseteq\alpha \}$
 because of 2-acyclicity.
 Hence, for all $\beta\subseteq E$
 \begin{align*}
  \beta\supseteq \alpha \quad &\Leftrightarrow \quad [v]_\beta\in \brck{w}\cap\brck{v} \\
  &\Rightarrow \quad [v]_\beta \in \rho(v,\gamma) \\
  &\Leftrightarrow \quad \beta \supseteq \gamma,
 \end{align*}
 which implies, in particular, $\gamma\subseteq\alpha$.
\end{proof}

We will use the mapping~$\rho$ to define generalisations
of $d(\brck{v},\brck{u})$ and $d(v,u)$. For 2-acyclic \ca
graphs, the sets $\brck{v}\cap\brck{u}$ and $\rho(v,\gamma)$
are generated, in some sense, by the single sets $\gen(v,u)$
and~$\gamma$, respectively. If~$\gamma$ is a subset of $\gen(v,u)$,
then $\rho(v,\gamma)$ is a superset of $\brck{v}\cap\brck{u}$ by
Lemma~\ref{le:tProp}. Hence, cutting out $\rho(v,\gamma)$
leaves a bigger hole in the dual hypergraph  and fewer
paths from $\brck{v}\setminus\rho(v,\gamma)$ to
$\brck{u}\setminus\rho(v,\gamma)$, and we can define a more
general measure of distance that is parametrized by $\rho(v,\gamma)$.

\begin{definition}\label{def:tDistanceDual}
Let $\mathcal{A}=(A,S)$ be a hypergraph and $t,X,Y\subseteq A$.
We denote with $d_t(X,Y)$ 
the distance between $X\setminus t$ and $Y\setminus t$
in the induced sub-hypergraph
$\mathcal{A}\setminus t := \mathcal{A}\upharpoonright (A\setminus t)$,
i.e.\ the graph-theoretic distance in its Gaifman graph.
\end{definition}

Essentially, we measure the length of the minimal paths that
go from one set to another and do \emph{not} go through a third
subset~$t$; we call such a path a \emph{non-$t$ path}.
The next step is to define the suitable analogon of non-$t$ paths
in \ca graphs. 
In Definition~\ref{def:tDistanceDual} we extended the set that is to
be avoided to the possibly larger set~$t$. Hence, the analogon on the side
of \ca graphs needs to avoid \emph{more} cosets as links which means
that we need to forbid a \emph{smaller} coset and all its supersets.

\begin{definition}
 Let~$G$ be a \ca graph, $v_1,v_\ellp$ two vertices, $\gamma$
 a set of generators and $t=\rho(v_1,\gamma)$. A coset path
 $ v_1,\alpha,v_2,\alpha_2,\dots,\alpha_\ell,v_\ellp $
 is a \emph{non-$t$ path} if, for all $1\leq i \leq \ell$,
 $$ [v_1]_\gamma \nsubseteq [v_i]_{\alpha_i}. $$
\end{definition}

Non-$t$ coset paths are a generalisation of
non-trivial coset paths (cf.\ Definition~\ref{def:cosetPath}) because
every non-trivial coset path from~$v$ to~$u$ is
a non-$t$ coset path, for $t=\rho(v,\gen(v,u))$.
Based on this generalisation, we can generalise
the former notion of distance to a
notion that depends on~$t$ in a straightforward manner.

\begin{definition}\label{def:tDistance}
Let~$G$ be a $2$-acyclic \ca graph, $v\neq u$ two vertices,
$\gamma\subseteq\Gamma$ and $t=\rho(v,\gamma)$.
The \emph{$t$-distance} $d_t(v,u)$ between~$v$ and~$u$ is defined as
the length of a minimal non-$t$ coset path from~$v$ to~$u$.
\end{definition}

\begin{remark}
 $t$-distance generalises the notion
 of distance from Definition~\ref{def:distance} in the sense that
 $d_t(v,u)=d(v,u),$ for $t=\rho(v,\gen(v,u))=\rho(u,\gen(v,u))$.
\end{remark}

\begin{remark}
 Depending on~$t$, $t$-distance allows for distance~1:
 $d_t(v,u)=1$ if and only if $[v]_{\gen(v,u)}\notin t$.
 However, the interesting cases are the ones where
 $\gamma\subseteq\gen(v,u)$, which implies
 $[v]_{\gen(v,u)}\in t$, for $t=\rho(v,\gamma)$.
\end{remark}

These two parametrized notions of distance,
$d_t(v,u)$ for \ca graphs and $d_t(\brck{v},\brck{u})$
for dual hypergraphs, are closely connected in the following
sense.

\begin{prop}\label{prop:twoDistances}
 For $\ell\geq 1$, let~$G$ be a sufficiently acyclic \ca graph,
 $v\neq u$ two vertices, $\gamma\subseteq E$ and $t=\rho(v,\gamma)$.
 Then 
 \[ d_t(v,u) = \ell
    \quad \Leftrightarrow \quad
    d_t(\brck{v},\brck{u}) = \ell -1.
 \]
\end{prop}

We give a formal proof in Section~\ref{sec:equivalencePaths}.
But first, we have a closer look at short non-$t$
coset paths and generalise some concepts from the previous
section about short coset paths.

\subsubsection{Short non-$t$ coset paths}\label{sec:shortNonTPaths}

In this section, we combine the notions of the set $\sh(v,u)$
and non-$t$ coset paths to obtain the parametrize operator
$\sh_t(\cdot,\cdot)$. It describes the direction one
has to take if one wants to move on a short non-$t$
coset path from~$v$ to~$u$, if such a path exists.

The zipper lemma implies the existence of $\sh(v,u)$.
For short non-$t$ coset paths
we need a specialized version of this operator.
As a reminder: $\alpha\subseteq E$ is a first
edge set for the pair of vertices $(v,u)$ if there
is a short coset path from~$v$ to~$u$ that starts
with an $\alpha$-edge.

\begin{definition}\label{def:shortT}
 Let~$G$ be a 2-acyclic \ca graph, $v,u\in G$
 and $\gamma\subseteq\gen(v,u)$ a set of generators.
 For $t=\rho(v,\gamma)$, we define the set of generators
 $\sh_t(v,z)$
 as the intersection of all the first generator sets of short
 non-$t$ coset paths from~$v$ to~$u$.
\end{definition}

In the definition of $\sh_t(v,u)$ we considered
a certain subset of all short coset paths
from~$v$ to~$u$.
If there are no such
paths, then this subset is
empty and $\sh_t(v,u)$ is not defined.
However, if there are short non-$t$ coset
paths $v,\alpha,\dots,u$ and $v,\beta,\dots,u$,
then there is a short
coset path $v,\alpha\cap\beta,\dots,u$
by Corollary~\ref{cor:edgeCut}, which
is also non-$t$ because $[v]_\gamma\nsubseteq [v]_\alpha$
and $[v]_\gamma\nsubseteq [v]_\beta$
imply $[v]_\gamma\nsubseteq [v]_{\alpha\cap\beta}$.
Thus, $\sh_t(v,u)$ is well-defined if short non-$t$
coset paths from~$v$ to~$u$ exist.

We continue with investigating the properties
of $\sh_t(v,u)$. This set behaves in a controlled and
intuitive manner in sufficiently acyclic graphs.
As $\sh_t(v,u)$ describes the direction of short
non-$t$ coset paths from~$v$ to~$u$, it changes
as one would expect if one moves to a neighbour~$v'$
of~$v$ via some $a$-edge: the direction for short
non-$t$ coset paths from~$v'$ to~$u$ necessarily
includes the generator~$a$.

\begin{lemma}\label{le:stepAwayT}
 Let $m\in\N$,
 $G$ be a \ca graph,
 $v,u$ two vertices,
 $\gamma\subseteq\gen(v,u)$ and
 $t=\rho(v,\gamma)$.
 Assume~$G$ is $2m+1$-acyclic,
 $d_t(v,u)\leq m$, and that there is
 $a\notin\sh_t(v,u)$
 such that $d_t(va,u)\leq m$,
 then $a\in\sh_t(va,u).$
\end{lemma}

\begin{proof}
 Let $\ell,k\leq m$, and
 $w_1,\alpha_1,\dots,\alpha_\ell,w_\ellp$ and
 $z_1,\beta_1,\dots,\beta_k,z_{k+1}$
 be two coset paths that avoid~$t$ with
 \begin{itemize}
  \item $w_1=z_1=u$, $w_\ellp=v$, $z_{k+1}=v'$, and
  \item $\alpha_\ell=\sh_t(v,u)$, $\beta_k=\sh_t(v',u)$.
 \end{itemize}
 Such paths exist by choice of~$v,u$ and~$v'$ and
 Definition~\ref{def:shortT}.
 If we assume $a\notin\sh_t(v',u)$, then $a\notin\alpha_\ell\cup\beta_k$.
 Together with
 $w_\ellp\notin[w_\ell]_{\alpha_\ellm\cap\alpha_\ell}$,
 $z_\kp\notin[z_k]_{\beta_\km\cap\beta_k}$ and $w_\ellp\neq z_{k+1}$
 this implies
 \begin{itemize}
  \item $[w_\ell]_{\alpha_\ellm\cap\alpha_\ell} \cap [w_\ellp]_{\alpha_\ell\cap\{a\}} = \emptyset$,
  \item $[w_\ellp]_{\alpha_\ell\cap\{a\}} \cap [z_{k+1}]_{\{a\}\cap\beta_k} = \emptyset$, and
  \item $[z_{k+1}]_{\{a\}\cap\beta_k} \cap [z_k]_{\beta_k\cap\beta_\km} = \emptyset$.
 \end{itemize}
 Hence,
 \[
  w_1,\alpha_1,w_2,\dots,w_\ell,\alpha_\ell,w_\ellp,a,z_\kp,\beta_k,z_k,\dots,z_2,\beta_1,z_1
 \]
 is a coset path of length $\ell+k+1\leq 2m+1$ from~$u$ to~$u$,
 which cannot exist by Lemma~\ref{le:cyclicZipper} in a
 $2m+1$-acyclic \ca graph.
\end{proof}

If we choose $\gamma=\gen(v,u)$ in the lemma above,
we obtain this special case:

\begin{corollary}\label{cor:stepAway}
 Let $m\in\N$,
 $G$ be a \ca graph and
 $v,u$ two vertices.
 Assume~$G$ is $2m+1$-acyclic,
 $d(v,u)\leq m$, and that there is $a\notin\sh(v,u)$
 such that $d(va,u)\leq m$, then $a\in\sh(va,u)$.
\end{corollary}

\subsubsection{Duality of paths}\label{sec:equivalencePaths}

In Section~\ref{sec:distance}, we claimed that
$d_t(v,u)$ and $d_t(\brck{v},\brck{u})$ are equivalent
(Proposition~\ref{prop:twoDistances})
although they are based on two seemingly very different
kinds of paths. In this section,
Lemmas~\ref{le:chordlessToCoset} and~\ref{le:cosetToChordless}
show a correspondence between non-$t$ coset paths and chordless
paths in $d(G)\setminus t$.
The former states that minimal paths in
$d(G)\setminus t$ induce non-$t$ coset paths.

\begin{lemma}\label{le:chordlessToCoset}
 Let~$G$ be a 2-acyclic \ca graph, $v_1\neq v_\ellp$ two vertices,~$\gamma$
 a set of generators and $t=\rho(v_\ellp,\gamma)$.
 Then a chordless path of length~$\ellp\geq 2$
 \[
  [v_1]_{\emptyset},\brck{v_1},[v_2]_{\alpha_1},\brck{v_2},[v_3]_{\alpha_2},\dots,[v_\ellp]_{\alpha_\ell},\brck{v_\ellp},[v_\ellp]_{\emptyset}
 \]
 in~$d(G) \setminus t$ from~$[v_1]_{\emptyset}$ to~$[v_\ellp]_{\emptyset}$
 induces a non-$t$ coset path
 \[
  v_1,\alpha_1,v_2,\dots,v_{\ell},\alpha_\ell,v_\ellp
 \]
 of length~$\ell$ in~$G$.
\end{lemma}

\begin{proof}
 Since $[v_\ip]_{\alpha_i}\in\brck{v_{i}}$ implies $v_{i}\in[v_\ip]_{\alpha_i}$, for all $1\leq i\leq\ell$,
 \[ v_1,\alpha_1,v_2,\dots,v_{\ell},\alpha_\ell,v_\ellp \] is a path in~$G$.
 First, we need to prove that it is also a coset path.
 If there is a vertex
 \[
  u\in[v_1]_{\emptyset\cap\alpha_1}\cap[v_2]_{\alpha_1\cap\alpha_2}
     =\{v_1\}\cap[v_1]_{\alpha_1}\cap[v_2]_{\alpha_2}
     =\{v_1\}\cap[v_1]_{\alpha_1}\cap[v_3]_{\alpha_2},
 \]
 then $u=v_1$ and $[v_3]_{\alpha_2}\in\brck{v_1}$, which implies
 that~$\brck{v_1}$ is a chord that connects~$[v_1]_\emptyset$
 and~$[v_3]_{\alpha_2}$; this cannot be because we assumed that
 the path is chordless. Analogously, one proves
 $[v_\ell]_{\alpha_{\ellm}\cap\alpha_\ell}\cap[v_\ellp]_{\alpha_\ell\cap\emptyset}=\emptyset$.
 If there is an $1< i\leq\ell$ and some vertex
  \[
   u\in[v_i]_{\alpha_{i-1}\cap\alpha_i}\cap[v_{i+1}]_{\alpha_i\cap\alpha_{i+1}}
   =[v_{i-1}]_{\alpha_{i-1}}\cap[v_i]_{\alpha_i}\cap[v_{i+1}]_{\alpha_{i+1}},
  \]
 then $[v_{i-1}]_{\alpha_{i-1}},[v_{i+1}]_{\alpha_{i+1}}\in\brck{v}$,
 which makes~$\brck{u}$ a chord for the path in~$d(G)$, contradicting chordlessness again.
 Second, the coset path is also non-$t$ because, for all $1 \leq i \leq \ell$,
 $$ [v_\ip]_{\alpha_i} \notin t \quad \Leftrightarrow \quad
    [v_\ellp]_\gamma \nsubseteq [v_\ip]_{\alpha_i}.
 $$
\end{proof}

Lemma~\ref{le:cosetToChordless} states the converse direction:
a minimal non-$t$ coset path in a \ca graph~$G$
induces a chordless path in~$d(G)\setminus t$.

\begin{lemma}\label{le:cosetToChordless}
 Let $\ell\geq 1$,~$G$ be a sufficiently acyclic \ca graph
 $v_1,v_\ellp$ two
 vertices, $\gamma\subseteq\gen(v_1,v_\ellp)$ a
 set of generators and $t=\rho(v_\ellp,\gamma)$.
 A non-$t$ coset path of length $\ell\geq 1$
 \[
  v_1,\alpha_1,v_2,\dots,v_{\ell},\alpha_\ell,v_{\ell+1}
 \]
 induces a chordless path of length~$\ell+1$
 \[
  [v_1]_{\emptyset},\brck{v_1},[v_2]_{\alpha_1},\brck{v_2},[v_3]_{\alpha_2},\dots,[v_\ellp]_{\alpha_\ell},\brck{v_\ellp},[v_\ellp]_{\emptyset}
 \]
 in~$d(G) \setminus t$.
\end{lemma}

\begin{proof}
 For all $1\leq i\leq\ell$, $v_{i}\in [v_\ip]_{\alpha_i}$
 implies $[v_\ip]_{\alpha_i}\in\brck{v_{i}}$, hence
 \[ [v_1]_{\emptyset},\brck{v_1},[v_2]_{\alpha_1},\brck{v_2},\dots,\brck{v_\ellp},[v_\ellp]_{\emptyset} \]
 is indeed a path in~$d(G)$.
 Furthermore, the coset path is non-$t$ because,
 for all $1 \leq i \leq \ell$,
 $[v_\ip]_{\alpha_i} \notin t$ if and only if
 $[v_\ellp]_\gamma \nsubseteq [v_\ip]_{\alpha_i}.$
 It remains to show that the path is chordless.
 
 Assume there is a chord, i.e.\ a hyperedge $\brck{v}\subseteq d(G)$
 that contains two vertices of the path in~$d(G)$ that have at
 least distance~2 on the path. Set~$\alpha_0=\emptyset$.
 If $\brck{v}$ contains~$[v_\ellp]_\emptyset$ and some vertex
 $[v_i]_{\alpha_\im}$, for $1\leq i\leq\ell$, then $v=v_\ellp$
 and $v_\ellp\in[v_i]_{\alpha_\im}$; this implies a short cyclic
 coset path from~$v_\ellp$ to~$v_\ellp$, which cannot exists
 in sufficiently acyclic \ca graphs by Lemma~\ref{le:cyclicZipper}.
 Otherwise, $\brck{v}$ contains two vertices $[v_i]_{\alpha_\im},[v_j]_{\alpha_\jm}$,
 for some $1\leq i,j\leq \ellp$ with $j>i+1$. Then
 $v\in[v_i]_{\alpha_\im}$ and $v\in[v_j]_{\alpha_\jm}$.
 The case $j=i+2$ and $v\in[v_i]_{\alpha_\im\cap\alpha_i}\cap[v_\ip]_{\alpha_i\cap\alpha_\ip}$
 (keep in mind that $[v_{i+2}]_{\alpha_\ip}=[v_{i+1}]_{\alpha_\ip}$)
 violates the coset cycle property. In any other case, we can find
 again a short cyclic coset path from~$v$ to itself.
\end{proof}

If a coset path is denoted as $v_1,\alpha_1,v_2,\dots,v_{\ell+1}$,
as in the lemma above, then $[v_\ip]_{\alpha_i}=[v_i]_{\alpha_i}$,
for all $1\leq i\leq\ell$.
Additionally, removing the first and last edge from a chordless
path does not change that it is chordless.

\begin{corollary}\label{cor:cosetToChordless}
 Let $\ell\geq 1$,~$G$ be a sufficiently acyclic \ca graph
 $v_1,v_\ellp$ two
 vertices, $\gamma\subseteq\gen(v_1,v_\ellp)$ a
 set of generators and $t=\rho(v_\ellp,\gamma)$.
 A non-$t$ coset path of length $\ell\geq 1$
 \[
  v_1,\alpha_1,v_2,\dots,v_{\ell},\alpha_\ell,v_{\ell+1}
 \]
 induces a chordless path of length~$\ell-1$
 \[
  [v_1]_{\alpha_1},\brck{v_2},[v_2]_{\alpha_2},\brck{v_3},\dots,\brck{v_\ell},[v_\ell]_{\alpha_\ell}
 \]
 in~$d(G) \setminus t$.
\end{corollary}

We can combine Proposition~\ref{prop:twoDistances} with the zipper lemma
and its implications to obtain a way to verify that the distance between
two vertices in a \ca graph or the distance between two hyperedges in its
dual hypergraph is long by looking only at a inner non-$t$ coset paths.
\begin{lemma}
 Let $\ell \geq 1$, $G$ be a sufficiently acyclic \ca graph, $v\neq u$
 two vertices, $\gamma\subseteq E$ and $t=\rho(v,\gamma)$.
 If there is no inner non-$t$ coset path from~$v$ to~$u$
 of length $\leq\ell$, then
 $$d_t(v,u)>\ell \quad \mbox{and} \quad d_t(\brck{v},\brck{u})>\ell -1. $$
\end{lemma}
\begin{proof}
 Assume $d_t(v,u)=k\leq\ell$, and let
 $ v_1,\alpha_1,v_2,\dots,v_k,\alpha_k,v_{k+1} $,
 with $v_1=v$ and $v_{k+1}=u$, be a non-$t$ coset path.
 Since~$G$ is sufficiently acyclic, this path is short.
 Hence,
 Corollary~\ref{cor:shortIsNontrivial}
 implies there are $v_i'\in[v_i]_{\alpha_i\cap\alpha}$,
 for $\alpha=\gen(v,u)$ and $1<i\leq k$, such that
 \[
  v_1,\alpha'_1,v_2',\alpha'_2,v_3',\dots,v_k',\alpha'_k,v_{k+1},
 \]
 for $\alpha'_i=\alpha_i\cap\alpha$, $1\leq i\leq k$,
 is a short inner coset path.
 This inner coset path is also non-$t$
 because $[v_\ellp]_{\gamma}\nsubseteq[v_i]_{\alpha_i}$ and
 $[v_i]_{\alpha_i\cap\alpha}=[v_i']_{\alpha_i\cap\alpha}\subseteq[v_i]_{\alpha_i\cap\alpha}$
 imply $[v_\ellp]_{\gamma}\nsubseteq [v_i']_{\alpha_i\cap\alpha}$.
 However, we assumed that such inner paths do not exist.
 Thus, $d_t(v,u)>\ell$ and by Corollary~\ref{prop:twoDistances}
 also $d_t(\brck{v},\brck{u}) > \ellm$.
\end{proof}

\section*{Conclusion}

This work provides a general toolbox for dealing with
the highly intricate overlap patterns of cosets in Cayley graphs
without short coset cycles.
These patterns are extremely dense, yet their highly regular structure
allows us to invoke notions of locality at multiple scales.
The overlap patterns are analysed in terms of related structures,
with a focus on the duality between \ca graphs and their associated dual hypergraphs.

We present several characterisations of local tree-likeness in \ca
structures, like the zipper lemma or regarding coset acyclic \ca graphs
as the dual image of $\alpha$-acyclic hypergraphs, which are locally
tree-decomposable. The zipper lemma gives us further insight into
the structure of coset paths on every level of granularity of
the coset overlap pattern. The duality between \ca graphs
and associated hypergraphs allows us to translate between coset paths
in \ca graphs and chordless graphs in the dual hypergraph.
Thus, we can translate problems in $n$-acyclic
\ca graphs to problems in $n$-acyclic hypergraphs and use
well-known results about $\alpha$-acyclicity to solve these problems.
Conversely,
we know how certain model-theoretic constructions on \ca graphs
impact their dual hypergraphs.
So far such techniques were successfully
applied in~\cite{CO17}, \cite{Ca18}, \cite{CO21} to characterise the
expressive power of Common Knowledge logic in certain classes of
Kripke structures that are based on \ca graphs.
This work makes these techniques accessible for a wider range of
applications.

\newpage

\bibliographystyle{abbrv}
\bibliography{S5}

\end{document}